\documentclass[12pt]{amsart}
\usepackage{amsmath, amssymb}
\let\temp\rmdefault
\usepackage{mathpazo}
\usepackage{mathrsfs}
\let\rmdefault\temp
\usepackage[margin=1in]{geometry}
\usepackage[amsmath]{empheq}
\usepackage{pdfrender,xcolor}
\usepackage{comment}
\usepackage{mathtools}
\usepackage{multicol}
\usepackage{enumitem, multicol}
\usepackage{setspace}
\usepackage[most]{tcolorbox}
\tcbset{colback=yellow!10!white, colframe=blue!50!blue, highlight math style= {enhanced, 
		colframe=blue,colback=red!10!white,boxsep=0pt}}
\usepackage[linkcolor=blue]{hyperref}
\hypersetup{colorlinks = true}
\newtheorem{theorem}{Theorem}[section]
\newtheorem{lemma}[theorem]{Lemma}

\newtheorem{proposition}[theorem]{Proposition}
\newtheorem{corollary}[theorem]{Corollary}

\newtheorem{definition}[theorem]{Definition}
\newtheorem{example}[theorem]{Example}

\newtheorem{Open Prob}[theorem]{Open Problem}
\newtheorem{remark}[theorem]{Remark}

\theoremstyle{definition}
\numberwithin{equation}{section}

\begin{document}
	
	
\title[A characterization of Banach spaces with numerical index one]{A characterization of Banach spaces with numerical index one}
\author{Subhadip Pal}
\address[Pal]{Subhadip Pal, Department of Mathematics, National Institute of Technology, Durgapur 713209, West Bengal, India.}
\email{palsubhadip2@gmail.com}
\author{Saikat Roy}
\address[Roy]{Saikat Roy, School of Advanced Sciences, VIT-AP University, Beside AP Secretariat, Amaravati, 522241, Andhra Pradesh, India.}
\email{saikatroy.cu@gmail.com, saikat.roy@vitap.ac.in}
\author{Debmalya Sain}
\address[Sain]{Debmalya Sain, Department of Mathematics, Indian Institute of Information Technology, Raichur 584135, Karnataka, India.}
\email{saindebmalya@gmail.com}
    
\keywords{Numerical range; Banach spaces with numerical index one; Dual space of operators; Extreme point; Support functional.\\
The research of Mr. Subhadip Pal is supported by the University Grants Commission, Government of India. Dr. Debmalya Sain acknowledges the financial support received from an ANRF-ECRG Project (ANRF/ECRG/2024/000436/PMS), titled ``Applications of the norm attainment problem in the geometry of Banach spaces and topological vector spaces"}

\begin{abstract}
We investigate the extremal properties of the unit ball of $L(X)_w^*$, the dual space of bounded linear operators defined on a Banach space $X$ equipped with the numerical radius norm. As an application of the present study, we obtain a geometric characterization of Banach spaces with numerical index one, which extends the well-known McGregor's characterization of finite-dimensional Banach spaces with numerical index one. We also present refinements of several earlier results in this direction, including an explicit description of the extreme points of $B_{L(X)_w^*}$, the unit ball of $L(X)_w^*$, for any finite-dimensional Banach space $X$. This allows us to obtain an independent and elementary proof of McGregor's characterization of finite-dimensional Banach spaces with numerical index one.
\end{abstract}
\maketitle
\section{Introduction}
\noindent
The numerical index is an important constant in the isometric theory of Banach spaces. In particular, Banach spaces with numerical index one enjoy several special geometric and analytic properties, which serve as the main motivation for a better understanding of such spaces. We refer the readers to the excellent monograph \cite{KMMP} for a detailed exposition of the current state-of-the-art on Banach spaces with numerical index one and their applications to the study of the so-called spear operators between Banach spaces. A complete characterization of finite-dimensional Banach spaces with numerical index one was obtained by McGregor in the seminal article \cite{MG}. Let us also point out that McGregor's characterization \cite[Theorem $3.1$]{MG}, by virtue of being completely geometric without involving operators, connects the study of Banach spaces with numerical index one with the underlying geometry of the space and its dual. In light of this finite-dimensional characterization, it became an important problem to find infinite-dimensional extensions of the same. We refer to \cite{KMMP, KMP, LMP}, and the references therein for more information on this topic. Although considerable progress has been made in this direction, an infinite-dimensional extension of  McGregor's characterization remains elusive. The main goal of this article is to complete this program by extending McGregor's characterization to the infinite-dimensional setting. Indeed, we completely characterize Banach spaces with numerical index one and obtain   McGregor's characterization as a direct consequence of our characterization. 

\medskip

\noindent
Let us now introduce the notation and terminology that will be used throughout this article.

\medskip

\noindent
Let $X$ be a Banach space and let $B_X$ and $S_X$ denote the closed unit ball and the unit sphere of $X,$ respectively. $E_{B_X}$ is the set of all extreme points of $B_X$. $X^*$ denotes the topological dual of $X$. Let $\theta$ be the zero vector in any vector space, except the scalar field. Throughout the article, $\mathbb{F}$ denotes the underlying scalar field, which can be either real or complex. For a nonzero vector $x$ in $X$, the collection of all \textit{support functional} at $x$ is denoted by $J(x)$, i.e., $J(x):=\{x^*\in S_{X^*}: x^*(x)=\|x\|\}$, which is always non-empty by the Hahn-Banach Theorem. We refer to $E_{J(x)}$ as the set of all extreme support functional at $x$. For any uni-modular scalar $\mu$, the section of support functionals at $x$ by $\mu$ is denoted as $J_\mu(x)$ and is defined by $J_\mu(x):=\{x^*\in S_{X^*}: x^*(x)=\mu\|x\|\}$. Evidently, $J_\mu(x)\neq \emptyset$ for any $\mu$, by the Hahn-Banach Theorem. For a fixed $\mu$, $J_\mu(x)$ is a convex and weak*-compact subset of $S_{X^*}$. Let $\psi :X\to X^{**}$ be the canonical embedding $x\mapsto \psi(x)$, where $\psi(x):X^*\to \mathbb{F}$ is given by $\psi(x)(g)=g(x)$ for each $ g \in X^* $ and for each fixed $x\in X$. Throughout this article, we identify $X$ with $\psi(X)$, whenever required. We further note that each such $ \psi(x) $ is a weak*-continuous functional on $ X^*.$ A \textit{face} of a convex set $U$ in a Banach space $ X $ is a non-empty subset of the form $\{x\in U: \operatorname{Re} x^*(x)=\sup_{U}(\operatorname{Re}x^*)\}$, where $x^*\in X^*$ is such that $\operatorname{Re}x^*$ attains its supremum on $U$.

\medskip

\noindent
For any Banach spaces $X$ and $Y$, let $L(X, Y)$($L(X)$, if $X=Y$) denote the space of all bounded linear operators from $X$ to $Y,$ endowed with the usual operator norm. Let $\text{Id}_X: X\to X$ denote the identity operator in $L(X)$. Whenever $ X $ is fixed and well-understood from the context, we simply write $ \text{Id}_X: X\to X =  \text{Id}. $ Given any $T\in L(X)$, we define 
\begin{equation}
W(T):=\{x^*(Tx): x^*\in S_{X^*}, x\in S_X, x^*(x)=1\},   
\end{equation}
\begin{equation}
w(T):=\sup\{|\xi|: \xi\in W(T)\}.   
\end{equation}
The quantities $W(T)$ and $w(T)$ are called the \textit{numerical range} and \textit{numerical radius} of $T$ respectively. The numerical radius $w(\cdot)$ is a seminorm on $L(X)$ and is equivalent to the operator norm if it defines a norm. In a complex Hilbert space $H$, $w(\cdot)$ always defines a norm on $L(H)$. Throughout this article, we assume that $w(\cdot)$ induces a norm on $L(X)$. The vector space $L(X),$ when endowed with the numerical radius norm, is denoted by $ L(X)_w. $

\medskip

\noindent
The concept of \textit{numerical index} of a Banach space $X$, denoted as $n(X)$, was first introduced by Lumer in 1968 (see \cite{DMPW}):
\[
n(X):= \inf\{w(T): T\in S_{L(X)}\}=\max\{k\geq 0: k\|T\|\leq w(T), T\in L(X)\}.
\]
Following the above definition, it is trivial to see that a Banach space $X$ has numerical index one if and only if $\|T\|=w(T)$, for any $T\in L(X)$. The study of Banach spaces with numerical index one is a deep and active direction of research, with important applications in operator theory. For some of the recent works on Banach spaces with numerical index one and the current state-of-the-art, the reader is referred to \cite{KMMP, KMMPQ, KMMS, KLMM, MMQRS, SPBB}. McGregor \cite{MG} obtained the following geometric characterization of \textit{finite-dimensional Banach spaces with numerical index one:}

\begin{theorem}\cite{MG}\label{MC}
Let $X$ be a finite-dimensional Banach space. Then $n(X)=1$ if and only if $|x^*(x)|=1$ for every $(x, x^*)\in E_{B_X}\times E_{B_{X^*}}$.
\end{theorem}

Geometric characterizations of (possibly infinite-dimensional) Banach spaces with numerical index one, in the spirit of the above result, are not known in the literature.  In this article, we obtain such a characterization, by studying the extreme points of $B_{L(X)_w^*}$. Our characterization involves the extreme points of the dual space of operators instead of operators themselves.

\medskip

\noindent
Let us now recollect a few standard definitions which are important for studying the numerical range of an operator on a Banach space. 
\begin{definition}\cite{KMMP}
Let $X$ be a Banach space and let $\mathscr{D}$ be a non-empty subset of $X$. $\mathscr{D}$ is said to be rounded if $S^1\mathscr{D}=\mathscr{D}$, where $S^1 :=\{\lambda\in \mathbb{F}:|\lambda|=1\}$ and $S^1\mathscr{D} :=\{\lambda y:\lambda\in S^1, y\in \mathscr{D}\}$. 
\end{definition}
\medskip
\noindent
For any non-empty subset $\mathscr{D}$ of a Banach space $ X, $ the closed convex hull of $\mathscr{D}$ is denoted by $\overline{\operatorname{co}}(\mathscr{D}).$ Similarly, the absolute convex hull of $\mathscr{D}$ is denoted by $\operatorname{aco}(\mathscr{D}).$ It is trivial to see that $\operatorname{aco}(\mathscr{D})=\operatorname{co}(S^1\mathscr{D})$.

\begin{definition}\label{Def1.4}\cite{KMMP} Let $\mathscr{D}, E$ be non-empty subsets of  $B_X$ and $X^*$, respectively. We say that $\mathscr{D}$ is \textit{norming} for $E$ if for every $f \in E$, $\|f\| = \sup \{|f(x)| : x \in \mathscr{D}\}$. Equivalently, this holds if $B_X = \overline{\operatorname{aco}}^{\sigma(X, E)}(\mathscr{D})$, where $\sigma(X, E)$ denotes the topology on $X$ induced by pointwise convergence of elements in $E$.
\end{definition}
\noindent
It is easy to observe that $\mathcal{D}\subseteq B_{X^*}$ is norming for $X$ if and only if $\|x\|=\displaystyle{\sup\{|f(x)|:f\in \mathcal{D}\}}$ for every $x\in X.$ Equivalently, $\mathcal{D} $ is norming for $X$ if and only if $ B_{X^*}=\overline{\operatorname{aco}}^{weak^*}(\mathcal{D})$.

\begin{definition}
Let $V$ be a non-empty closed convex subset of a Banach space $X$. A supporting hyperplane to $V$ is a hyperplane that contains $V$ in one of its closed half-spaces and intersects $V$ with at least one point. An element $x$ in the boundary of $V$ is said to be an exposed point of $V$ if there exists a hyperplane of support $\mathcal{H}$ to $V$ such that $\mathcal{H}\cap V=\{x\}$.
\end{definition}
The structure of this article is as follows: Apart from the introductory section, the article is divided into two main sections. The first section concerns the extreme points of $B_{L(X)_w^*},$ for a Banach spaces $X,$  offering a refinement of \cite[Theorem $2.1$]{M}. As the main highlight of this article, we present a geometric characterization of Banach spaces with numerical index one, which extends the classical McGregor's characterization of finite-dimensional Banach spaces with numerical index one. Next, we completely determine the extreme points of $B_{L(X)_w^*},$ for any finite-dimensional Banach space $X$, which improves \cite[Theorem $2.3$]{M}. Furthermore, this also provides an elementary alternative proof of McGregor's characterization. Additionally, we present a counting formula for the extreme points of $B_{L(X)_w^*},$ for any finite-dimensional real polyhedral Banach space $X$. 
	
\section{Extreme points of the unit ball in the dual of certain operator spaces under the numerical radius norm}
\noindent
In this section, we characterize the Banach spaces with numerical index one by studying the extreme points of $ B_{L(X)_w^*}. $ For any $x\in X$ and any $x^*\in X^*$, let us define $x^*\otimes x:L(X)_{w}\to \mathbb{F}$ by $[x^*\otimes x](T) = x^*(Tx),$ for each $ T \in L(X)_{w}. $ We begin with the following preliminary results on the extremal structure of $ B_{L(X)_w^*}.$ Later, we will observe in Remark \ref{Th1 remark} that the converse part of the following result is also true.

\begin{lemma}\label{BL2.1}
Let $X$ be a Banach space. Suppose that $ M \subseteq B_{L(X)_w^*} $ is non-empty, rounded, and norming for $L(X)_w$. Then 
\[
B_{L(X)_w^*} = \overline{\operatorname{co}}^{weak^*}(M).
\]
\end{lemma} 
Before proving our next result, let us mention the following well-known fact that the extreme points of $ B_{L(X)_w^*} $ are contained in the weak*-closure of a particularly convenient subset of $ X^*\otimes X: $

\begin{theorem}\cite[Theorem $2.1$]{M}\label{AM2.2}
Let $X$ be a Banach space. Then $E_{B_{L(X)_w^*}}\subseteq \overline{\mathscr{A}}^{weak^*}$, where
\[
\mathscr{A}:=\{x^*\otimes x: x^*\in B_{X^*}, x\in B_X, |x^*(x)|=\|x^*\|\|x\|\}.
\]
\end{theorem}

 We next present a refinement of the above result by considering only the extreme points of $ B_{X^*}, $ and thus replacing the set $ \mathscr{A} $ by a smaller set $ \mathscr{M}. $  It is worth mentioning in this context that our proof is completely different from the one given in \cite{M}, where the bipolar theorem was used. 
 
\begin{theorem}\label{Th1}
Let $X$ be a Banach space. Then $E_{B_{L(X)_{w}^*}}\subseteq \overline{\mathscr{M}}^{weak^*}$, where
\[
\mathscr{M}:=\{x^*\otimes x: x^*\in E_{B_{X^*}}, x\in S_X, |x^*(x)|=1\}.
\]
\end{theorem}
\begin{proof}
We first show that $\mathscr{M}$ is non-empty. Consider any $x\in S_X$. It follows from the Hahn-Banach Theorem that $J(x)$ is a non-empty weak*-compact, convex subset of $B_{X^*}$. Therefore, by the Krein-Milman Theorem $J(x)$ has an extreme point, say $x_0^*$. We claim that $x_0^*\in E_{B_{X^*}}$. Suppose on the contrary that $x_0^*\notin E_{B_{X^*}}$, then there exist $\lambda\in (0, 1)$ and $x_1^*, x_2^*\in B_{X^*}\setminus \{x_0^*\}$ such that $x_0^*=\lambda x_1^*+(1-\lambda)x_2^*$. It is evident that $x_1^*(x)=x_2^*(x)=1$. Consequently, $x_1^*, x_2^*\in J(x)$ and $x_0^*$ fails to be an extreme point of $J(x)$, a contradiction. Thus, $x_0^*\otimes x\in\mathscr{M}$ and it is non-empty. Next, we show that $\mathscr{M}$ is norming for $L(X)_w$. Clearly, for any $T\in L(X)_w$,
\begin{equation}\label{NT2.2}
\begin{aligned}
w(T)&=\sup\{|[x^*\otimes x](T)|: x^*\in S_{X^*}, x\in S_X, |x^*(x)|=1\}\\
& = \sup\{|[x^*\otimes x](T)|: x\in S_X,~ x^*\in J_\mu(x),~\mu\in S^1\}.
\end{aligned}
\end{equation}
Observe that for a fixed $\mu_0\in S^1$, the collection $J_{\mu_0}(z)$ is a face of $B_{X^*}$. For a fixed pair $(x_0, \mu_0)\in S_X\times S^1$, we now show that 
\[
\sup\{|[x^*\otimes x_0](T)|: x_0\in S_X,~ x^*\in J_{\mu_0}(x_0)\}=\sup\{|[x^*\otimes x_0](T)|: x_0\in S_X,~ x^*\in E_{J_{\mu_0}(x_0)}\}.
\]
Let
\[
\alpha_1 := \sup \left\{ \left| [x^* \otimes x_0](T) \right| : x_0 \in S_X, \, x^* \in J_{\mu_0}(x_0) \right\}
\]
and 
\[
\alpha_2 := \sup \left\{ \left| [x^* \otimes x_0](T) \right| : x_0 \in S_X, \, x^* \in E_{J_{\mu_0}(x_0)} \right\}.
\]
Clearly, $\alpha_2 \leq \alpha_1$. We claim that $\alpha_2 \geq \alpha_1$. If possible, let $\alpha_2 < \alpha_1$. By applying the Krein-Milman Theorem on the weak*-compact convex set $J_{\mu_0}(x_0)$, we have
\[
J_{\mu_0}(x_0) = \overline{\operatorname{co}}^{weak^*} (E_{J_{\mu_0}(x_0)}).
\]
Therefore, for any $x^* \in J_{\mu_0}(x_0)$, there exists a net $\{ x^*_{\beta} \}_{\beta \in \Lambda}$ in $\operatorname{co} (E_{J_{\mu_0}(x_0)})$, say, $x^*_{\beta} = \sum_{i=1}^{k_{\beta}} \lambda_i^{(\beta)} x_{i}^{*^{(\beta)}}$, where each $x_{i}^{*^{(\beta)}} \in E_{J_{\mu_0}(x_0)}$, and the coefficients satisfy $1 \geq \lambda_i^{(\beta)} \geq 0,$ with $\sum_{i=1}^{k_{\beta}} \lambda_i^{(\beta)} = 1, \quad \text{for each } \beta \in \Lambda$, such that $x^*_{\beta} \xrightarrow[\beta]{\text{weak}^*} x^*$. It follows that $|x^*_{\beta}(Tx_0)| \to |x^*(Tx_0)|$.
However, $\sup\{|y^*(Tx_0)|:y^*\in E_{J_{\mu_0}(x_0)}\} < \alpha_1$. Therefore, $\sup\{|y^*(Tx_0)|:y^*\in \operatorname{co}(E_{J_{\mu_0}(x_0)})\} < \alpha_1$ and consequently, $\sup\{|y^*(Tx_0)|:y^*\in \overline{\operatorname{co}}^{weak^*}(E_{J_{\mu_0}(x_0)})\} < \alpha_1$. This leads to a contradiction, since 
\[
\overline{\operatorname{co}}^{weak^*} (E_{J_{\mu_0}(x_0)})=J_{\mu_0}(x_0)\quad \text{and}\quad\sup \left\{ \left| [x^* \otimes x_0](T) \right| : x_0 \in S_X, \, x^* \in J_{\mu_0}(x_0) \right\}=\alpha_1.
\]
Thus, $\alpha_2\geq \alpha_1$ and consequently, $\alpha_1=\alpha_2$. Therefore, in continuation with the equality \eqref{NT2.2}, it follows that
\[
w(T)= \sup\{|[x^*\otimes x](T)|: x\in S_X,~ x^*\in E_{J_{\mu}(x)}, \mu\in S^1\}.
\]
Since $E_{J_{\mu}(z)} \subset E_{B_{X^*}}$ for any $(z, \mu) \in S_X \times S^1$, it is easy to see that 
\begin{align*}
w(T)&= \sup\{|[x^*\otimes x](T)|: x\in S_X,~ x^*\in E_{B_{X^*}}, |x^*(x)|=1\}\\
&= \sup\{|[x^*\otimes x](T)|: x^*\otimes x\in \mathscr{M}\}.
\end{align*}
It follows that for any $x^*\otimes x\in \mathscr{M}$, $|[x^*\otimes x](T)|\leq w(T)$ and consequently, $\mathscr{M}\subseteq B_{L(X)_w^*}$. Thus, $\mathscr{M}$ is norming for $L(X)_w$. Evidently, $\mathscr{M}$ is rounded. Therefore, it follows from Lemma \ref{BL2.1} that $B_{L(X)_{w}^*}=\overline{\operatorname{co}}^{weak^*}(\mathscr{M}).$ Then, the desired conclusion $E_{B_{L(X)_{w}^*}}\subseteq \overline{\mathscr{M}}^{weak^*}$ follows directly from \cite[Theorem $7.8$]{C}, thereby completing the proof.
\end{proof}

Additionally, if $ X $ is reflexive then $ B_X $ is weakly compact. Therefore, by applying the Krein-Milman Theorem on $ B_X $ in the same way as before, we can further strengthen Theorem \ref{Th1}.

\begin{corollary}\label{reflexive}
Let $X$ be a reflexive Banach space. Then 
\[
E_{B_{L(X)_{w}^*}}\subseteq \overline{\{x^*\otimes x: x^*\in E_{B_{X^*}}, x\in E_{B_X}, |x^*(x)|=1\}}^{weak^*}.
\]
\end{corollary}
We can directly derive the following remark using a similar technique as in the proof of Theorem \ref{Th1}.

\begin{remark}\label{Th1 remark}
Suppose that $M \subseteq B_{L(X)_w^*} $ is non-empty and rounded. Then we can say that $B_{L(X)_w^*}=\overline{\operatorname{co}}^{weak^*}(M)$ if and only if $M$ is norming for $L(X)_w$. Indeed, if $M$ is norming for $L(X)_w$ then by Lemma \ref{BL2.1}, we have $B_{L(X)_w^*}=\overline{\operatorname{co}}^{weak^*}(M)$. To prove the converse implication, we note that for any $T\in L(X)_w,$ we have
\begin{align*}
w(T)&=\sup\{|f(T)|:f\in B_{L(X)_w^*}\}\\
&=\sup\{|f(T)|:f\in \overline{\operatorname{co}}^{weak^*}(M)\}\\
&=\sup\{|f(T)|:f\in M\},
\end{align*}
where the last equality follows from a similar technique used in the proof of Theorem \ref{Th1}. This establishes the desired conclusion. Furthermore, in a similar way, we can conclude that $B_{L(X)^*}=\overline{\operatorname{co}}^{weak^*}(M)$ if and only if $M \subseteq B_{L(X)^*} $ is non-empty, rounded, and norming for $L(X)$.
\end{remark}

As an application to Theorem \ref{Th1}, it is possible to refine the necessary part of \cite[Theorem $3.4$]{M}. Before stating Theorem $3.4$ of \cite{M}, let us first mention the concept of Birkhoff-James orthogonality \cite{B, J}. For $x, y\in X$, we say that $x$ is \textit{Birkhoff-James orthogonal} to $y,$ written as $ x \perp_B y, $ if $\|x+\lambda y\|\geq \|x\|$ for all scalars $\lambda$. For any subspace $E$ of $X$, we say that $x\perp_B E$ if $x\perp_B y$ for all $y\in E$. Birkhoff-James orthogonality on $L(X)$ has been widely explored (see \cite{G, MPS, TR}), with important applications to the study of differentiability properties in operator spaces. For $T, A \in L(X)_w$, we say that $T$ is \textit{Birkhoff-James orthogonal} to $A$ \textit{with respect to the numerical radius norm}, denoted by $T\perp_w A,$ if for all scalars $\lambda$, the inequality  
\[
w(T + \lambda A) \geq w(T)
\]  
holds. We refer the reader to the recent article \cite{RS} for some applications of numerical radius Birkhoff-James orthogonality.

\begin{theorem}\cite[Theorem $3.4$]{M}\label{Th2.6}
Let $X$ be a Banach space, and let $\mathscr{W}$ be an $n$-dimensional subspace of $L(X)_w$. Suppose $T \in L(X)_w$ with $T \neq 0$. Then $T \perp_w \mathscr{W}$ if and only if the following conditions hold.  
\begin{itemize}
\item[(a)] There exist positive scalars $c_1, c_2, \dots, c_m > 0$ ($m\leq n+1$ if $\mathbb{F}=\mathbb{R}$ and $m\leq 2n+1$ if $\mathbb{F}=\mathbb{C}$) such that $\sum_{i=1}^{m} c_i = 1$.
\item[(b)] For each $1 \leq i \leq m$, there exists a net $\{x_{i\beta}^* \otimes x_{i\beta}\}_\beta$ in $\mathscr{A}$, where $\mathscr{A}$ is same as mentioned in Theorem \ref{AM2.2}, satisfying $\lim_{\beta} x_{i\beta}^*(Tx_{i\beta}) = w(T)$ and  
\[
\sum_{i=1}^{m} c_i \lim_{\beta} x_{i\beta}^*(Ax_{i\beta}) = 0, \quad \forall A \in \mathscr{W}.
\]
\end{itemize}
\end{theorem}

In light of Theorem \ref{Th1}, the following refinement of the necessary part of Theorem \ref{Th2.6} is rather straightforward to prove. To avoid repetition of arguments, we omit the proof and invite the reader to verify the details.

\begin{theorem}
Let $X$ be a Banach space, and let $\mathscr{W}$ be an $n$-dimensional subspace of $L(X)_w$. Suppose $T \in L(X)_w$ with $T \neq 0$. Then $T \perp_w \mathscr{W}$ if and only if the following conditions hold.  
\begin{itemize}
\item[(a)] There exist positive scalars $c_1, c_2, \dots, c_m > 0$ ($m\leq n+1$ if $\mathbb{F}=\mathbb{R}$ and $m\leq 2n+1$ if $\mathbb{F}=\mathbb{C}$) such that $\sum_{i=1}^{m} c_i = 1$.
\item[(b)] For each $1 \leq i \leq m$, there exists a net $\{x_{i\beta}^* \otimes x_{i\beta}\}_\beta$ in $\mathscr{M},$ where $\mathscr{M}$ is same as mentioned in Theorem \ref{Th1}, satisfying $\lim_{\beta} x_{i\beta}^*(Tx_{i\beta}) = w(T)$ and  
\[
\sum_{i=1}^{m} c_i \lim_{\beta} x_{i\beta}^*(Ax_{i\beta}) = 0, \quad \forall A \in \mathscr{W}.
\]
\end{itemize}
\end{theorem}

We are now ready to present the highlight of this article, which is the characterization of the Banach spaces $ X $ having numerical index one. The main importance of this characterization stems from the fact that it is an infinite-dimensional extension of McGregor's characterization of finite-dimensional Banach spaces having numerical index one. Indeed, we obtain McGregor's characterization as a consequence of the general characterization proved in this article. We would like to point out that our approach focuses on the extreme points of the unit ball $B_{L(X)^*},$ and does not explicitly involve operators. It is important to note that our result is \textit{in the same spirit as that of McGregor \cite{MG}}.  It is worth mentioning in this context that in \cite{LMP}, Lopez et al. obtained the following sufficient condition for $n(X)=1$.

\begin{proposition}\cite{LMP}
Let $X$ be an infinite-dimensional Banach space. If $|\chi(x^*)|=1$ for every $x^*\in E_{B_{X^*}}$ and every $\chi\in E_{B_{X^{**}}}$ then $n(X)=1$.
\end{proposition}

It is known that the above condition is not necessary for $ 
n(X) = 1, $ see \cite[Remark $4.2(c)$]{KMMS}. On the other hand, the following necessary conditions for $ n(X) = 1, $ were obtained in \cite{LMP}.

\begin{proposition}\cite[Lemma $1$]{LMP}
Let $X$ be a Banach space with numerical index one. Then
\begin{itemize}
\item[i)] $|\chi(x^*)|=1$ for every $\chi\in E_{B_{X^{**}}}$ and every weak*-denting point $x^*\in B_{X^*}$.
\item[ii)] $|x^*(x)|=1$ for every $x^*\in E_{B_{X^*}}$ and every denting point $x\in B_X$. 
\end{itemize}
\end{proposition}

Let us also mention here that in the same paper \cite{LMP}, it has been shown that a reflexive real Banach space with numerical index one must be finite-dimensional. It is an open problem whether the same result also holds true for complex Banach spaces. In addition to this, it has been proved in \cite{DMPW} that M-spaces, L-spaces, and their isometric preduals have numerical index one. We refer to \cite{KMP} for more information on this topic. Before presenting the promised characterization of Banach spaces with numerical index one, we recall the definition of a spear element \cite{A, KMMP} in a Banach space $X$. An element $z\in X$ is called a \textit{spear element} if for every $x\in X$ there is a modulus one scalar $t$ for which $\|z+tx\| = 1+\|x\|$ holds. 

\begin{theorem}\label{Infinite}
Let $X$ be a Banach space. Then the following statements are equivalent
\begin{itemize}
\item[(a)] $n(X)=1$.
\item[(b)] $B_{L(X)^*}=\overline{\operatorname{co}}^{weak^*}\{x^*\otimes x: x^*\in E_{B_{X^*}}, x\in S_X, |x^*(x)|=1\}$. 
\item[(c)] For any $f\in E_{B_{L(X)^*}}$ there exists a net $\{x_\beta^*\otimes x_\beta\}_{\beta\in\Lambda}$, where $x_\beta^*\in E_{B_{X^*}}$, $x_\beta\in S_X$ and $|x_\beta^*(x_\beta)|=1$, such that $x_\beta^*\otimes x_\beta\xrightarrow{weak^*} f$ in $L(X)^*$.
\item[(d)] For any $f\in E_{B_{L(X)^*}}$ there exists a net $\{x_\beta^*\otimes x_\beta\}_{\beta\in\Lambda}$, where $x_\beta^*\in S_{X^*}$, $x_\beta\in S_X$ and $|x_\beta^*(x_\beta)|=1$, such that $x_\beta^*\otimes x_\beta\xrightarrow{weak^*} f$ in $L(X)^*$.
\end{itemize}
\end{theorem}
\begin{proof}
$(a)\implies (b):$ Assume that $n(X)=1$. Then, for any $T\in L(X)$, $\|T\|=w(T)$. Therefore, it follows from Theorem \ref{Th1} that $\mathscr{M}\subseteq B_{L(X)^*} $ is non-empty, rounded, and norming for $L(X)$ as well. Consequently, by Remark \ref{Th1 remark} we have $B_{L(X)^*}=\overline{\operatorname{co}}^{weak^*}(\mathscr{M})$.

\medskip

$(b)\implies (c):$ Let $B_{L(X)^*}=\overline{\operatorname{co}}^{weak^*}\{x^*\otimes x: x^*\in E_{B_{X^*}}, x\in S_X, |x^*(x)|=1\}$. Then $E_{B_{L(X)^*}}\subseteq \overline{\{x^*\otimes x: x^*\in E_{B_{X^*}}, x\in S_X, |x^*(x)|=1\}}^{weak^*}$, by \cite[Theorem $7.8$]{C} and consequently, $(c)$ holds.

\medskip

$(c)\implies (d):$ Follows trivially.

\medskip

$(d)\implies (a):$ Assume that for any $f\in E_{B_{L(X)^*}},$ there exists a net $\{x_\beta^*\otimes x_\beta\}_{\beta\in\Lambda}$, where $x_\beta^*\in S_{X^*}$, $x_\beta\in S_X$ and $|x_\beta^*(x_\beta)|=1$, such that $x_\beta^*\otimes x_\beta\xrightarrow{weak^*} f$ in $L(X)^*$. Therefore, for any $T\in L(X)$, $[x_\beta^*\otimes x_\beta](T)\to f(T)$. In particular, $[x_\beta^*\otimes x_\beta](\text{Id})\to f(\text{Id})$ and consequently, $|x_\beta^*(x_\beta)|\to |f(\text{Id})|$. Since for each $\beta\in \Lambda$, $|x_\beta^*(x_\beta)|=1$, it follows that $|f(\text{Id})|=1$. Thus, from Proposition $3.2$ of \cite{KMMP}, we conclude that $\text{Id}$ is a spear element of $L(X)$. Again, from Proposition $1.1$ of \cite{KMMP} we know that: A Banach space $X$ has numerical index one if and only if Id is a spear element of $L(X)$. This establishes that $n(X)=1$ and finishes the proof.
\end{proof}

Our next goal is to show that McGregor's characterization of finite-dimensional Banach spaces having numerical index one can be obtained as a consequence of the above theorem. We begin by recalling the following fundamental characterization from \cite{LO}. Prior to this work, the corresponding result in the setting of real Banach spaces was established in \cite[Theorem $1.3$]{RSL}.

\begin{theorem}\cite[Theorem $1$]{LO}\label{Lima}
Let $K(X, Y)$ be the space of all compact operators from a Banach space $X$ to a Banach space $Y$, either both real or both complex. Then
\[
E_{B_{K(X, Y)^*}}=E_{B_{X^{**}}}\otimes E_{B_{Y^*}}.
\]

\end{theorem}

We are now ready to deduce McGregor's characterization from our main result Theorem \ref{Infinite}.

\begin{corollary}\cite[Theorem $3.1$]{MG}
Let $X$ be a finite-dimensional Banach space. Then $n(X)=1$ if and only if $|x^*(x)|=1$ for every $(x, x^*)\in E_{B_X}\times E_{B_{X^*}}$.
\end{corollary}

\begin{proof}
To prove the necessary part, consider any $(x^*, x)\in E_{B_{X^*}}\times E_{B_X}$. It follows from Theorem \ref{Lima} that $x^*\otimes x\in E_{B_{L(X)^*}}$. We claim that $|x^*(x)|=1$. Applying Theorem \ref{Infinite}, there exists a sequence $\{x_n^*\otimes x_n\}$ with $|x_n^*(x_n)|=1$ for each $n\in \mathbb{N}$, such that $ x_n^*\otimes x_n\to x^*\otimes x. $ It follows that
\[ [x_n^*\otimes x_n](\text{Id})\to [x^*\otimes x](\text{Id}) \implies  |x_n^*(x_n)|\to |x^*(x)|. \]
Since $|x_n^*(x_n)|=1$, for each $n\in \mathbb{N}$, therefore, $|x^*(x)|=1$. Conversely, for the sufficient part, let us assume that for any $(x^*, x)\in E_{B_{X^*}}\times E_{B_X}$, it holds that $|x^*(x)|=1$. Considering the constant sequence $\{x^*\otimes x\},$ it follows from Theorem \ref{Infinite} that $n(X)=1$. This establishes the result.
\end{proof}

The tensor product $X \otimes Y$ of Banach spaces $X$ and $Y$, equipped with the projective norm, is denoted by $X \otimes_\pi Y$, and its completion is written as $X \widehat{\otimes}_\pi Y$. Let us denote the space of all bounded bilinear forms $\mathcal{B}$ on $X \times Y$ by  $\mathscr{B}(X,Y)$ ($\mathscr{B}(X)$ whenever $X=Y$), which is a Banach space with norm
\[
\|\mathcal{B}\| := \sup\{|\mathcal{B}(x, y)| : x \in B_X, y \in B_Y\}.
\]
From \cite{BF}, we have the identification
\begin{equation}\label{tensor}
(Y \widehat{\otimes}_\pi X)^* \cong \mathscr{B}(Y, X) \cong L(Y, X^*) \cong L(X, Y^*).
\end{equation}

\medskip

Let $Y$ be a reflexive Banach space and let $Y^*=X.$ We can define an isometric isomorphism $\phi_0:(X^* \widehat{\otimes}_\pi X)^{*}\to L(X)$ given by
\[ \phi_0(u)=T_u, \qquad\forall~u\in (X^* \widehat{\otimes}_\pi X)^{*}, \] 
where $T_u(x)(x^*)=u(x^*\otimes x).$ Since $X$ is reflexive, we have the identification $L(X) = L(X, X^{**})$. Next, define $\phi_1:L(X)^*\to (X^* \widehat{\otimes}_\pi X)^{**},$ given by 
\[
\phi_1(g)(u)=g(\phi_0(u)), \qquad\forall~u\in (X^* \widehat{\otimes}_\pi X)^{*}. 
\]
Since $\phi_0$ is an isometric isomorphism, it is not difficult to see that $\phi_1$ is also an isometric isomorphism. By continuing with the same notations, the above discussion essentially leads to the following corollary to Theorem \ref{Infinite}.

\begin{corollary}
Let $X$ be a reflexive Banach space. Then the following statements are equivalent
\begin{itemize}
\item[(a)] $n(X)=1$.
\item[(b)] For any $(g\circ \phi_0)\in E_{B_{(X^* \widehat{\otimes}_\pi X)^{**}}}$, there exists a net $\{[x_\beta^*\otimes x_\beta]\circ \phi_0\}_{\beta\in\Lambda}$, where $x_\beta^*\in E_{B_{X^*}}$, $x_\beta\in S_X$ and $|x_\beta^*(x_\beta)|=1$, such that $[x_\beta^*\otimes x_\beta]\circ \phi_0\xrightarrow{weak^*} (g\circ \phi_0)$ in $(X^* \widehat{\otimes}_\pi X)^{**}$.
\item[(c)] For any $(g\circ \phi_0)\in E_{B_{(X^* \widehat{\otimes}_\pi X)^{**}}}$, there exists a net $\{[x_\beta^*\otimes x_\beta]\circ \phi_0\}_{\beta\in\Lambda}$, where $x_\beta^*\in S_{X^*}$, $x_\beta\in S_X$ and $|x_\beta^*(x_\beta)|=1$, such that $[x_\beta^*\otimes x_\beta]\circ \phi_0\xrightarrow{weak^*} (g\circ \phi_0)$ in $(X^* \widehat{\otimes}_\pi X)^{**}$.
\end{itemize}
\end{corollary}
We next derive the exact expression of the extreme points of the unit ball $B_{L(X)_w^*},$ for a finite-dimensional Banach space $X$. We note that a subset relation for the same was previously established in \cite{M}. On the other hand, our approach not only revisits this relationship from a different perspective but also establishes the converse inclusion, thereby improving  the following Theorem.

\begin{theorem}\cite[Theorem $2.3$]{M}
Let $X$ be a finite-dimensional Banach space. Then 
\[
E_{B_{L(X)_{w}^*}}\subseteq \{x^*\otimes x: x^*\in E_{B_{X^*}}, x\in E_{B_X}, |x^*(x)|=1\}.
\] 
\end{theorem}
We require the following preliminary observation for our purpose, the proof of which is omitted because it is rather straightforward.

\begin{lemma}\label{dual ball}
Let $ X $ be a Banach space equipped with two norms $ \|\cdot\|_1 $ and $ \|\cdot\|_2 $ such that $ \|x\|_1 \leq \|x\|_2 $ for all $ x \in X $. If $ \|\cdot\|_i^* $ denotes the dual norm corresponding to $ \|\cdot\|_i $ for $ i = 1, 2 $, then
\[
\|f\|_2^* \leq \|f\|_1^*, \quad \forall~ f \in X^*.
\]
\end{lemma}

\begin{theorem}\label{completely continuous}
Let $X$ be a finite-dimensional Banach space. Then
\[
E_{B_{L(X)_{w}^*}}=\{x^*\otimes x: x^*\in E_{B_{X^*}}, x\in E_{B_X}, |x^*(x)|=1\}.
\] 
\end{theorem}
\begin{proof}
We first consider the set $\mathscr{A}$ mentioned in Theorem \ref{AM2.2}. Observe that for any $T\in L(X)_w$, $w(T)=\sup\{|[x^*\otimes x](T)|:x^*\otimes x\in \mathscr{A}\}$. Therefore, $\mathscr{A}\subseteq B_{L(X)_w^*}$ is non-empty and it is norming for $L(X)_w$. Evidently, $\mathscr{A}$ is rounded, and therefore, by Lemma \ref{BL2.1} we have $B_{L(X)_w^*}=\overline{\operatorname{co}}(\mathscr{A})$. It now follows from \cite[Theorem $7.8$]{C} that $E_{B_{L(X)_{w}^*}}\subseteq \overline{\mathscr{A}}$. We show that $\mathscr{A}$ is compact. It is enough to show that $\mathscr{A}$ is closed. Consider a sequence $(x_n^*\otimes x_n)_{n\in \mathbb{N}}$ in $\mathscr{A}$ such that $x_n^*\otimes x_n\to f$. Now, sequences $(x_n^*)_{n\in \mathbb{N}}$ and $(x_n)_{n\in \mathbb{N}}$ both has convergent subsequences in $B_{X^*}$ and $B_X,$ respectively. If necessary, passing through a suitable subsequence of natural numbers, we can assume that $x_n^*\to x^*\in B_{X^*}$ and $x_n\to x\in B_{X}$. We now show that $f=x^*\otimes x$. For any $T \in L(X)_w$, the continuity of $T $ ensures that $[x_n^* \otimes x_n](T) \to [x^* \otimes x](T). $ Thus, $x_n^*\otimes x_n\to x^*\otimes x,$ and consequently, $|x_n^*(x_n)|\longrightarrow |x^*(x)|$. Since $|x_n^*(x_n)|=\|x_n^*\|\|x_n\|\longrightarrow\|x^*\|\|x\|$, it follows that $|x^*(x)|=\|x^*\|\|x\|$. Thus, $\mathscr{A}$ is closed, as desired. Therefore, we obtain $E_{B_{L(X)_{w}^*}}\subseteq \mathscr{A}$. Now, for any $x^*\otimes x\in E_{B_{L(X)_{w}^*}}$, it is easy to see that $(x^*, x)\in S_{X^*}\times S_X$. We claim that $(x^*, x)\in E_{B_{X^*}}\times E_{B_X}$. On the contrary, suppose that $x^*\notin E_{B_{X^*}}$. Therefore, there exist $t\in (0, 1)$ and $x_1^*, x_2^*\in B_{X^*}$ such that $x^*=tx_1^*+(1-t)x_2^*$, where $x^*\neq x_1^*, x_2^*$. Now, 
\[
1=|x^*(x)|\leq t|x_1^*(x)|+(1-t)|x_2^*(x)|\leq 1
\]
implies that $|x_1^*(x)|= 1$ and $|x_2^*(x)|= 1$. Thus, $x_1^*\otimes x, x_2^*\otimes x\in \mathscr{A}\subseteq B_{L(X)_w^*}$. Moreover, $x^*\otimes x= t[x_1^*\otimes x]+(1-t)[x_2^*\otimes x]$. However, $x^*\otimes x\in E_{B_{L(X)_{w}^*}}$ implies that $x^*\otimes x= x_1^*\otimes x=x_2^*\otimes x$. Therefore, we obtain from \cite[Lemma $2.2$]{M} that $x^*=x_1^*=x_2^*$ and consequently, $x^*\in E_{B_{X^*}}$. Similarly, we get $x\in E_{B_X}$. Therefore,
\begin{equation}\label{eq2.4}
\begin{aligned}
E_{B_{L(X)_{w}^*}}&\subseteq \{x^*\otimes x: x\in E_{B_X}, x^*\in E_{B_{X^*}}, |x^*(x)|=1\}\\
&\subseteq \{x^*\otimes x: x\in E_{B_X}, x^*\in E_{B_{X^*}}\}\\
&=E_{B_{L(X)^*}}\quad(\text{by using \cite[Theorem $1$]{LO} and the reflexivity of}~X).
\end{aligned}
\end{equation} 
Thus, we are only left to prove that $\{x^*\otimes x: x\in E_{B_X}, x^*\in E_{B_{X^*}}, |x^*(x)|=1\}\subseteq E_{B_{L(X)_{w}^*}}$. Assume that $x_0^*\otimes x_0\in \{x^*\otimes x: x\in E_{B_X}, x^*\in E_{B_{X^*}}, |x^*(x)|=1\}\subseteq B_{L(X)_{w}^*}$. By $(\ref{eq2.4}),$ we have $x_0^*\otimes x_0\in E_{B_{L(X)^*}}$. If possible, let $x_0^*\otimes x_0\notin E_{B_{L(X)_{w}^*}}$. Therefore, there exists $t\in (0, 1)$ and $f_1,\, f_2\in B_{L(X)_w^*}$ such that $x_0^*\otimes x_0=tf_1+(1-t)f_2$, where $x_0^*\otimes x_0\neq f_1,\, f_2$. Also, using Lemma \ref{dual ball}, we have $B_{L(X)_{w}^*}\subseteq B_{L(X)^*}$ and consequently, $f_1,\, f_2\in B_{L(X)^*}$. However, it implies that $x_0^* \otimes x_0 \notin E_{B_{L(X)^*}}$, leading to a contradiction. Therefore, $x_0^*\otimes x_0\in E_{B_{L(X)_{w}^*}}$ and consequently, $\{x^*\otimes x: x\in E_{B_X}, x^*\in E_{B_{X^*}}, |x^*(x)|=1\}\subseteq E_{B_{L(X)_{w}^*}}$. This completes the proof. 
\end{proof}

The above theorem, apart from being interesting its own right, also motivates us to present an alternative proof of McGregor's characterization.

\medskip

\noindent
\textbf{An alternative proof of Theorem \ref{MC}:}
Let $n(X)=1$. Then $B_{L(X)^*}= B_{L(X)_{w}^*}$, and consequently, $E_{B_{L(X)^*}}=E_{B_{L(X)_{w}^*}}$. Now, by applying \cite[Theorem $1$]{LO}, we get 
\[
E_{B_{L(X)^*}}=\{x^*\otimes x:x^*\in E_{B_{X^*}}, x\in E_{B_X}\}=E_{B_{L(X)_{w}^*}}.
\]
Therefore, by Theorem \ref{completely continuous}, it follows that  $|x^*(x)|=1,$ establishing the necessary part. Conversely, let for any $(x, x^*)\in E_{B_X}\times E_{B_{X^*}}$, $|x^*(x)|=1$. Thus, $x^*\otimes x\in E_{B_{L(X)_{w}^*}}$. It follows from \cite[Theorem $1$]{LO} that $x^*\otimes x\in E_{B_{L(X)^*}}$. Therefore, $E_{B_{L(X)_{w}^*}}=E_{B_{L(X)^*}}$. Thus, by applying the Krein-Milman Theorem on $B_{L(X)_w^*}$ and $B_{L(X)^*},$ we get $B_{L(X)_w^*}=B_{L(X)^*}$. We note that for any $T\in L(X)_w$, 
\[ w(T) = \sup\{|f(T)|:f\in B_{L(X)_w^*}\} = \sup\{|f(T)|:f\in B_{L(X)^*}\} = \|T\|. \]
This proves that $n(X)=1,$ finishing the proof.

\medskip

If $X$ is a real polyhedral Banach space, Theorem \ref{completely continuous} provides a counting formula for the extreme points of $B_{L(X)_w^*}$.
\begin{remark}
Let $X$ be an $n$-dimensional real polyhedral Banach space. For any $x_0\in E_{B_X}$, the number of extreme functionals attaining norm at $x_0$ is given by $2|E_{J(x)}|$. Thus, by virtue of Theorem \ref{completely continuous}, we have
\[
|E_{B_{L(X)_{w}^*}}|=2\displaystyle{\sum_{x\in E_{B_X}}}|E_{J(x)}|.
\]
\end{remark}
\begin{example}
Let $X=\ell_\infty^n(\mathbb{R})$. Using the well-known identification of $X^*$ with $\ell_1^n(\mathbb{R}),$ it is easy to see that $|E_{J(x)}|=n$, for each $x\in E_{B_X}$. Since $|E_{B_X}|=2^n,$ we have 
\[
|E_{B_{L(X)_{w}^*}}|=n2^{n+1}.
\]
\end{example}
As another application of the extremal study conducted in this article, we next characterize the exposed points of $B_{L(X)_w^*}$, for finite-dimensional Banach space $X$, using the so-called nu-smooth operators. We denote the numerical radius attainment set of $T$ by $M_{w(T)},$ which is defined as 
\begin{equation}\label{NAS}
M_{w(T)} :=\{(x, x^*)\in S_X\times S_{X^*}: x^*(x)=1, |x^*(Tx)|=w(T)\}.  
\end{equation}    
We denote by $\mathbb{M}_f$  the norm attainment set of a functional $f\in S_{L(X)_w^*}$, defined as
\[
\mathbb{M}_f :=\{A\in B_{L(X)_w}:f(A)=w(A)=1\}.
\]
It is clear that if $f$ is an exposing functional for some $A\in B_{L(X)_w}$ then $\mathbb{M}_f$ is a singleton. Let us also recall that a non-zero element $x\in S_X$ is said to be \textit{smooth} if $J(x)$ is a singleton. The space $X$ is said to be smooth if each non-zero $x\in S_X$ is smooth. We record the following elementary observation.

\begin{proposition}\label{exposed points}
Let $X$ be a reflexive Banach space. Then $f\in B_{X^*}$ is an exposed point of $B_{X^*}$ if and only if $J(x)=\{f\}$, for some smooth point $x\in S_X$.
\end{proposition}

For any $T\in L(X)_w$, let $J_w(T)$ denote the collection of all \textit{support functional(s)} at $T$ \textit{with respect to numerical radius norm}, defined by
\[
J_w(T):=\{f\in S_{L(X)_w^*}:f(T)=w(T)\}.
\]
For any non-zero $T\in L(X)_w$, $T$ is nu-smooth if and only if $J_w(T)$ is singleton.

\medskip

A Banach space $X$ is said to be \textit{strictly convex} if $E_{B_{X}}= 
S_X.$ It is rather easy to observe that $X$ is strictly convex if and only if every point of $ S_X $ is an exposed point (in particular, an extreme point) of $ B_X. $ In general, an extreme point of a convex set need not be an exposed point, even in the finite-dimensional case. We end this article with the following result which ensures that for a finite-dimensional smooth strictly convex Banach space $X$, all the extreme points of $B_{L(X)_w^*}$ are exposed points.

\begin{theorem}
Let $X$ be a finite-dimensional strictly convex smooth Banach space. Then, all the extreme points of $B_{L(X)_w^*}$ are exposed points.
\end{theorem}
\begin{proof}
We have from Theorem \ref{completely continuous} that
\[
E_{B_{L(X)_{w}^*}}=\{x^*\otimes x: x^*\in E_{B_{X^*}}, x\in E_{B_X}, |x^*(x)|=1\}.
\]
Let $x^*\otimes x\in E_{B_{L(X)_{w}^*}}$ be arbitrarily chosen. Then $(x, x^*)\in E_{B_X}\times E_{B_{X^*}}$ such that $x^*(x)=\mu$, for some $\mu\in S^1$. We claim that $x^*\otimes x$ is an exposed point of $B_{L(X)_w^*}$. To prove our claim, we next construct a $T_\mu(x,x^*)\in L(X)_w$ that is nu-smooth and $J_w(T_\mu(x,x^*))=\{x^*\otimes x\}$. For any $y\in X$, we define 
\[
T_\mu(x,x^*)(y)=\overline{\mu}x^*(y)\overline{\mu}x.
\]
We now show that $M_w(T_\mu(x,x^*))=\left\{\left (\overline{\lambda}~ x, \overline{\overline{\lambda}\mu}~ x^*\right ):\lambda\in S^1\right\}$. Let $(\Tilde{y}, \Tilde{y^*})\in M_w(T_\mu(x,x^*))$. Then by \eqref{NAS} we have $(\Tilde{y}, \Tilde{y^*})\in S_X\times S_{X^*}$ with $\Tilde{y^*}(\Tilde{y})=1$ and $|\Tilde{y^*}(T_\mu(x,x^*)\Tilde{y})|=w(T_\mu(x,x^*))$. It is easy to see that $w(T_\mu(x,x^*))=1$, and 
\[
1=|\Tilde{y^*}(T_\mu(x,x^*)\Tilde{y})|=|\Tilde{y^*}(x^*(\Tilde{y})x)|=|\Tilde{y^*}(x)||x^*(\Tilde{y})|
\]
if and only if 
\[
|\Tilde{y^*}(x)|=1=|x^*(\Tilde{y})|.
\]
Assume that $\Tilde{y^*}(x)=\lambda$ and $x^*(\Tilde{y})=\lambda'$, where $\lambda, \lambda'\in S^1$. It follows from the smoothness of $X$ that $\overline{\lambda}\Tilde{y^*}=\overline{\mu}x^*$ and using strict convexity of $X$ we have $\overline{\lambda'}\Tilde{y}=\overline{\mu}x$. Now, $\Tilde{y^*}(\Tilde{y})=1$ implies that $\lambda'=\overline{\lambda}\mu$. Thus,
\[
M_w(T_\mu(x,x^*))\subseteq\left\{\left (\overline{\lambda}~ x, \overline{\overline{\lambda}\mu}~ x^*\right ):\lambda\in S^1\right\}.
\]
However, $\left\{\left (\overline{\lambda}~ x, \overline{\overline{\lambda}\mu}~ x^*\right):\lambda\in S^1\right\}\subseteq M_w(T_\mu(x, x^*))$, as $(x, \overline{\mu}x^*)\in E_{B_X}\times E_{B_{X^*}}$ and $\overline{\mu}x^*(x)=1$. Consequently, $M_w(T_\mu(x,x^*))=\left\{\left (\overline{\lambda}~ x, \overline{\overline{\lambda}\mu}~ x^*\right ):\lambda\in S^1\right\}$, and by \cite[Theorem $2.5$]{RS}, $T_\mu(x,x^*)$ is nu-smooth. Now, 
\[
[x^*\otimes x](T_\mu(x,x^*))=x^*(T_\mu(x,x^*)(x))=\overline{\mu}x^*(x)\overline{\mu}x^*(x)=1=w(T_\mu(x,x^*)).
\]
Also, by our assumption, we have $x^*\otimes x\in S_{L(X)_{w}^*}$. Therefore, it follows that $x^*\otimes x\in J_w(T_\mu(x,x^*))$.  Since $T_\mu(x,x^*)$ is nu-smooth, $J_w(T_\mu(x,x^*))=\{x^*\otimes x\}$. Thus, Proposition \ref{exposed points} implies that $x^*\otimes x$ is an exposed point of $B_{L(X)_w^*}$. This completes the proof. 
\end{proof}

\bibliographystyle{plain}

\end{document}